\documentclass[a4paper,12pt]{article}

\usepackage{times, url}
\DeclareUnicodeCharacter{3B1}{\ensuremath{\alpha}}

\usepackage{amsmath,amssymb,amsthm,amsfonts,bm,float,cite}
\usepackage{booktabs}

\usepackage{xcolor,colortbl}
\definecolor{Mauve}{rgb}{0.88,0.69,1}
\definecolor{Golden}{rgb}{0.99,0.76,0.0}
\definecolor{Carribean}{rgb}{0.4,0.8,0.67}
\definecolor{LightCyan}{rgb}{0.88,1,1}

\usepackage{tabularx}
\usepackage{boldline}
\usepackage[utf8]{inputenc}

\newtheorem{definition}{Definition}[section]
\newtheorem{remark}{Remark}[section]
\newtheorem{theorem}[definition]{Theorem}
\newtheorem{lemma}[definition]{Lemma}
\newtheorem{corollary}[definition]{Corollary}
\newtheorem{conjecture}{Conjecture}[section]
\newtheorem{proposition}{Proposition}[section]

\usepackage[none]{hyphenat}

\usepackage[center]{caption}

\usepackage{mathtools}

\usepackage[pdftex]{graphicx}

\makeatletter
\def\smalloverbrace#1{\mathop{\vbox{\m@th\ialign{##\crcr\noalign{\kern3\p@}%
  \tiny\downbracefill\crcr\noalign{\kern3\p@\nointerlineskip}%
  $\hfil\displaystyle{#1}\hfil$\crcr}}}\limits}
\makeatother

\usepackage{mathtools} 

\makeatletter

\usepackage{mathtools}

\DeclareGraphicsExtensions{.png}

\makeatletter
\renewcommand{\@biblabel}[1]{[#1]\hfill}
\makeatother

\begin{document}
\nocite{*}
\setcounter{page}{1}

\begin{center}
{\LARGE \bf  The congruence speed formula\\[4mm] } 
\vspace{8mm}

{\Large \bf Marco Rip\`a}
\vspace{3mm}

World Intelligence Network\\ 
Rome, Italy\\
e-mail: \url{marco.ripa@mensa.it}
\vspace{2mm}

\end{center}
\vspace{10mm}


\noindent \sloppy {\bf Abstract:} We solve a few open problems related to a peculiar property of the integer tetration ${^{b}a}$, which is the constancy of its congruence speed for any sufficiently large $b=b(a)$. Assuming radix-$10$ (the well-known decimal numeral system), we provide an explicit formula for the congruence speed $V(a) \in \mathbb{N}_0$ of any $a \in \mathbb{N}-\{0\}$ that is not a multiple of $10$. In particular, for any given $n \in \mathbb{N}$, we prove to be true Rip\`a’s conjecture on the smallest $a$ such that $V(a)=n$. Moreover, for any $a \neq 1 : a \not\equiv 0 \pmod {10}$, we show the existence of infinitely many prime numbers $p_j:=p_j(V(a))$ such that $V(p_j)=V(a)$.


\vspace{5mm}


\section{Introduction} \label{sec:Intr}
The aim of this paper is to give a general formula for the ``congruence speed" of tetration
\cite{Googology:3, OEIS:13}, affirmatively answering the final conjecture stated in \cite{Ripa:16}. The properties that arise from our study \cite{Yan:20} are valid for many different numeral systems \cite{Germain:2, Urroz:19}, but (from here on out) we assume radix-$10$.

First of all, let us introduce the constancy of the congruence speed of the integer tetration ${^{b}a}$.

\begin{definition} \label{def1.1}
Let $a \in \mathbb{N}-\{0,1\}$ not be a multiple of $10$. Let $d \in \mathbb{N}$. The power tower of height
$b \in \mathbb{N}-\{0\}$ represents the integer tetration
 ${^{b}a} := \begin{cases}
     a         \quad    \quad  \quad   \textnormal{if}  \quad b=1 \\
     a^{(^{b-1}a)}  \quad       \textnormal{if}  \quad      b \geq 2 \,
    \end{cases}
$. 

Given ${^{b-1}a} \equiv {^{b}a} \pmod{10^d} \wedge {^{b-1}a} \not\equiv {^{b}a} \pmod{10^{d+1}}, \forall b>a \geq 2, V(a,b)$
returns the strictly positive integer such that
${^{b}a} \equiv {^{b+1}a} \pmod{10^{d+V(a)}} \wedge {^{b}a} \not\equiv {^{b+1}a} \pmod{10^{d+V(a)+1}}$,

and we define $V(a,b)$ as the ``congruence speed" of the base $a$ at the given height of its hyperexponent $b$. Consequently, if $a=2$, the tetrations for $b$ from $1$ to $5$ are ${^{1}2}=2$, ${^{2}2}=4$, ${^{3}2}=16$, ${^{4}2}=65536$, and ${^{5}2}=\dots19156736$ (respectively), so we can see that $V(2,1)=V(2,2)=0$, while $V(2,3)=V(2,4)=1$.
\end{definition}

Now, let us assume $a \in \mathbb{N} : a \not\equiv 0 \pmod {10}$ in the rest of the paper.

Since it is known \cite{Ripa:16} that $b-1 \geq a \geq2$ is $a$ sufficient but not necessary condition for $V(a,b)=V(a)$, let $b : b>a \geq 2$ be (unless differently specified) such that we can simply indicate as $V(a)$ the ``constant congruence speed" of $a$, where $V(a)$ has been already defined in Reference \cite{Ripa:16}, Definition \ref{def1.2}, assuming $V(1)=0$ (see \cite{Ripa:16}, pages 248–249). To this purpose, it is crucial to underline that the constancy of the congruence speed of $a$ is a general property concerning also cases where the minimum value of $b$ such that $V(a,b)=V(a)$ is smaller than $a$ itself (for $a$ proof that $b \geq 2$ implies $V(3,b)=V(3)$, see \cite{Ripa:16}, Lemma \ref{Lemma1}). Furthermore, for given pivotal tetrations, an in-depth analysis of the smallest $b$ such that the related congruence speed is constant can be found in Reference \cite{Ripa:15}.


\section{A formula for the constant congruence speed of \bm{$a$}} \label{sec:Section2}
In the present Section we study $V(a)$, taking into account every $a \not\equiv 0 \pmod {10}$ \cite{OEIS:13}. In the first subsection, for any given $V(a)=n \in \mathbb{N}-\{0,1\}$, we show which are the smallest bases whose residues modulo $10$ cover the whole set $\{1,2,3,4,5,6,7,8,9\}$. The second subsection is devoted to provide a general formula which maps any $a$ whose constant congruence speed is given, for any $V(a)\in \mathbb{N}$.

\subsection{Finding bases with arbitrarily large \bm{$V(a)$} in the ring of decadic integers} \label{sec:SUBSECTION 2.1}

In order to describe the structure of $V(a) \in \mathbb{N}-\{0\}$ in radix-$10$, for any $a \not\equiv 0 \pmod {10}$, it can be useful to move the problem on $\mathbb{Z}_{10}$, the ring of $10$-adic integers.

\begin{proposition} \label{prop1}
The $10$-adic integers form a commutative ring, and we indicate it as $\mathbb{Z}_{10}$ \cite{Lubin:4}.
\end{proposition} 

\begin{proposition} \label{prop2}
Any positive integer can be represented as a $10$-adic integer $\alpha$. $\alpha$ can be written as an infinitely long string of digits going to the left of a ``fixed digit". The aforementioned fixed digit, that we indicate as $s_1$, is the one which defines the congruence class (AKA residue modulo $10$) of the corresponding base of the tetration ${^{b}a}$. For any $n=1,2,3,\dots$, let us consider $S(n) := s_n\_s_{(n-1)}\_\dots\_s_2\_s_1 \in \frac{\mathbb{Z}}{10^n \mathbb{Z}}$, where the underscore symbol has been introduced in order to indicate the juxtaposition of nonnegative integers so that $S(n+1)=s_{(n+1)}\_S(n)$. The residues modulo $10^n$ satisfy the congruence relation $S(n) \equiv S(n+1) \pmod {10^n}$. Now, assume $s_1 \in \{1,2,3,\dots,9\}$ and, if $n \geq 2$, let $s_{j+1} \in \{0,1,2,\dots,9\}$ for every $j \in \{1,\dots,n-1\}$. In particular, we have that $a_{s_1} (n) := \sum_{j=0}^{n-1} s_{j+1} \cdot 10^j \Rightarrow a_{s_1} (n) \equiv s_1 \pmod {10}$. Thus, given $n$, $a_{s_1} (n)$ is a strictly positive decimal integer, smaller than $10^{n+1}$, having $s_1$ as its least significant digit.

On the other hand, we know that, $ \forall a_{s_1}(n)$, $\exists α \in \mathbb{Z}_{10}$ such that $α=\sum_{j=0}^{+\infty} s_{j+1} \cdot 10^j \equiv \sum_{j=0}^{n-1}s_{j+1} \cdot 10^j \pmod {10^n}=a_{s_1} (n)$.
\end{proposition} Consequently, the idea to work with decadic integers can be an efficient approach to solve (radix-$10$) the problem of finding, for each congruence class modulo $10$ belonging to the set $\{1,3,7,9\}$, which is the smallest tetration base whose constant congruence speed is equal to any given positive integer $n$.

\begin{definition} \label{def1.2}
For every $n \in \mathbb{N}-\{0\}$, we define $\tilde{a}(n) := \min_n \{a \not\equiv 0 \pmod {10} : V(a)=n\}$. In addition, for any given $s_1 \in \{1,2,3,4,5,6,7,8,9\}$, let us denote with $a_{\{s_1 \}}(n)$ the generic element of the set $\textnormal{A}_{s_1}(n) :=\{a : a \equiv s_1 \pmod {10} \wedge V(a)=n\}$. Consequently, $\forall n \geq 1$, $\tilde{a}_{s_1}(n)= \min_n\left(\textnormal{A}_{(s_1)}(n)\right)$ and $\tilde{a}_{\{1,2,3,4,5,6,7,8,9\}} (n)=\min_n \{\tilde{a}_1(n),\tilde{a}_2(n),\dots,\tilde{a}_9(n)\}=\tilde{a}(n)$.
\end{definition}

In order to avoid notational misunderstandings, let us specify that, from here on, $a_{\{c,d\}} \in \{\textnormal{A}_c \cup \textnormal{A}_d \}$ refers to every (generic) tetration base that is congruent modulo $10$ to $c$ or $d$ (assuming that $c$ and $d$ represent two distinct elements of the set $\{1,2,3,4,5,6,7,8,9\}$). We use the notation $a_{[c,d]}$  to indicate that we are considering one particular element from the congruence class $c$ modulo $10$ and also another one from the congruence class $d$ modulo $10$ so that $\tilde{a}_{[c,d]}$  (see Section \ref{sec:SUBSECTION 2.2}) returns the smallest base which is congruent modulo $10$ to $c$ and the smallest one which is congruent modulo $10$ to $d$, while $\tilde{a}{\{c,d\}} = \min(\tilde{a}_c,\tilde{a}_d)$ gives the smallest base that is congruent modulo $10$ to $c$ or $d$. In particular, let us simply write $a_{s_1}$ (omitting brackets) if, by selecting each one of the allowed congruence classes $s_1$, we always get a unique base, making it clear that the elements belonging to special subsets of $\{a \in \mathbb{N} : a \not\equiv 0 \pmod {10}\}$ will be uniquely marked by adding symbols on the top of $a$ itself (such as the aforementioned $\tilde{a}_c$ or even ${\smalloverbrace{a}}^{*}$), while different mathematical objects will be introduced by using other letters so that $y_c(5)$, which indicates the $c$-th solution in $\mathbb{Z}_{10}$ of the equation $y^5=y$ (see Proposition \ref{prop6}), should not be confused with any base ending in $c$ (the decadic integer originated by $y_c(5)$ is in no way forced to have $c$ as its rightmost digit). To this purpose, we finally observe that 
$y^5=y$ returns at most two decadic integers, say $\alpha'_c$ and $\alpha''_c$ , both having the same $c$ as their rightmost digit; since each of the $α'_{s_1}$ is well defined for any given $s_1=1,2,\dots,9$, we are free to introduce some general properties pertaining to the $α'_{s_1}$ without needing to add unnecessary brackets.

\begin{proposition} \label{prop3}
Let us consider the standard decimal numeral system (radix-$10$). It follows that the corresponding $g$-adic ring that we have to take into account is the decadic one ($g=10$) \cite{Mahler:5}, but $10$ is not a prime number or a power of a prime (since $10=2 \cdot 5=p_1 \cdot p_2$, $p_1 \neq p_2$). Thus, for every odd $s_1$ (as defined in Proposition \ref{prop2}), we can find more than one polymorphic $\alpha = \dots \_s_1$ that arises when we solve in $\mathbb{Z}_{10} := \displaystyle\lim_{\longleftarrow}\frac{\mathbb{Z}}{10^n \mathbb{Z}}$ (i.e., the set of formal series $\sum_{j=0}^{+\infty} s_{j+1} \cdot 10^j$, $s_{j+1} \in \{0,1,2,3,4,5,6,7,8,9\}$) the fundamental equation $y^t = y$.

Therefore, assuming $s_{n+1} \neq 0$, $\forall s_1 \in \{1,3,5,7,9\}$, we can find two order-$n$ residues of as many polymorphic integers (i.e., $\alpha' \neq \alpha''$ such that $\alpha' \equiv \alpha'' \pmod {10}$) whose expansions modulo $10^n$ are always characterized by a constant congruence speed equal to $n$ (e.g., $s_1=7 \Rightarrow \alpha'_{s_1}=\dots 66295807$ and $\alpha''_{s_1}=\dots 92077057$ both satisfy $y^5=y$, and $n=7$ implies that $V(\alpha'\pmod {10^7})=V(6295807)=V(\alpha''\pmod {10^7})= V(2077057)=7$ since the eighth rightmost digit of $\alpha'_7$ and $\alpha''_7$ is not zero).
\end{proposition}

\begin{conjecture} \label{conj1}
Let the tetration base $a$ be greater than $1$. Let $\textnormal{len}(a) \in \mathbb{N}-\{0\} : 10^{\textnormal{len}(a)-1} \leq a < 10^{\textnormal{len}(a)}$  denote the number of digits of $a$. If $a \not\equiv \{0,3,7\} \pmod {10}$, then $b \geq \textnormal{len}(a)+2$ is a sufficient condition for $V(a,b)=V(a)$.
\end{conjecture}

\begin{remark} \label{rem1}
Assuming $a \equiv \{2,3,4,6,8,9,11,12,13,14,16,17,19,21,22,23\} \pmod {25}$, by Reference \cite{Ripa:16}, Hypothesis 1, $V(a)=1$. This confirms the statement of Conjecture \ref{conj1} for any $a$ as above, since we know that $V(a) \geq 1 \wedge V(a,b+1) \leq V(a,b)$ holds for any $b \geq 3$ \cite{Ripa:15, Ripa:16}; as a clarifying example of the property $V(a,b+1) \leq V(a,b)$ extended to nontrivial congruence classes modulo $25$, we can take a look at \cite{Ripa:15}, page 27, which includes the $\textit{phase shift analysis}$ of the base $143^{625}$ congruent to $18$ modulo $25$, explaining why $V(143^{625},1)=0 \wedge V(143^{625},2)=V(143^{625},3)=6 \wedge V(143^{625},4)=5 \wedge V(143^{625}, \hspace{1mm} b : b \geq 5)=V(a)=4$ occurs. \linebreak Thus, $a : V(a)=1$ implies that, for any $b \geq 3$, $1=V(a,b) \geq V(a,b+1) \geq 1$ so that $V(a,b+1)=V(a,b)=1$ (e.g., $V(2, \hspace{1mm} b : b \geq 3)=V(2, \hspace{1mm} b : b \geq \textnormal{len}(2)+2)=1$ is consistent with the expected result \cite{Googology:3}).
\end{remark}

\begin{proposition} \label{prop4}
The constant congruence speed of $a$ is well defined if and only if $a \not\equiv 0 \pmod {10}$ \cite{Ripa:16}. In particular, $V(a) \geq 1 \Rightarrow a \geq 2$, and $b \geq a+1$ represents a sufficient, but not a necessary, condition for the constancy of the congruence speed of $a$. Moreover, $\exists a^* \equiv \{3,7\}\pmod {10} : V(a^*,2 \leq b \leq \textnormal{len}(a^* )+2) = 1+V(a^*,b \geq \textnormal{len}(a^*)+3) = 1+V(a^*)$, where $a^* \in \{807,81666295807,81907922943,\dots\}$, and this follows from Proposition \ref{prop6} (check ($i \in \{3,4,9,10\},n>2$) in Equation (\ref{eq:2}) such that, picking each of the four aforementioned values of $i$ so that $\alpha_{\{3,7\}}$  is given, $\frac{\alpha_{\{3,7\}} \pmod{10^{n+1}} - \alpha_{\{3,7\}} \pmod {10^n}}{10^n }=5$).
\end{proposition}

\begin{proposition} \label{prop5}
$g=10=2 \cdot 5=p_1 \cdot p_2 \Rightarrow \textnormal{gcd}(p_1,p_2)=1$ (see Proposition \ref{prop3}). Since in $\mathbb{Z}_{10}$ (which is not an integral domain) $\exists h \neq 0 \wedge r \neq 0$ such that $h \cdot r=0$, it follows that, for every $n \in \mathbb{N}$, $5^{2^n } \cdot 2^{5^n} \equiv 0 \pmod {10^n}$ by the ring homomorphism $\phi : \mathbb{Z}_{10} \rightarrow \frac{\mathbb{Z}}{10^n \mathbb{Z}}$. Since the sequence $\{5^{2^n }\}_n := {5^{2^0 }},{5^{2^1 }},{5^{2^2 }},\dots$ converges $5$-adically to $0$ and $2$-adically to $1$, and $\{2^{5^n} \}_{\infty}=1-\{5^{2^n }\}_{\infty}$, the above is the unique pair which induces the decomposition of $\mathbb{Z}_{10}$. Thus, $\mathbb{Z}_{10}\cong \mathbb{Z}_5 \oplus \mathbb{Z}_2$ (where $\oplus$ indicates the direct sum) since, for $p$ prime, the complete ring $\mathbb{Z}_p$ contains only the two idempotents elements $0$ and $1$, and the $5$-adically plus $2$-adically convergence implies the $10$-adically convergence (by Cauchy’s convergence criterion). Hence, assume $h(n) \simeq 5^{2^n }$ and $r(n) \simeq 2^{5^n }$ in order to solve the fundamental equation $y^t=y$, introduced by Proposition \ref{prop3}.

Given $s_1=5$, if $h_n=5^{2^n} \pmod {10^n}$, then $\displaystyle\lim_{\infty \leftarrow n} h_n=\dots 92256259918212890625$ \cite{Lubin:4}.

Similarly, for $s_1=2$, $r_n=2^{5^n} \pmod {10^n} \Rightarrow 
\displaystyle\lim_{\infty \leftarrow n}
r_n=\dots804103263499879186432$.

Now, let $\{y_i(t),i=1,2,\dots \}$ be the set of the $i$ solutions in $\mathbb{Z}_{10}$ of $y^t=y$, and also let $\{y_{\hat{i}}(t),\hat{i}=1,2,\dots \}$ be a subset of $\{y_i(t)\}$. If $t=2$, then $\nexists \hat{i} : y_{\hat{i}}(t) \in \{0,1\} \Leftrightarrow y_{\hat{i}}(2) \in \{h,1-h\}$ for any $\hat{i}$, so let $y_1(2)=h$ and $y_2(2)=1-h$.

Following the path above, it is possible to verify that all the solutions of $y^t=y$ belong to the set of the solutions of $y^5=y$ \cite{Michon:6}. Thus, for every given $\hat{i}$ such that $y_{\hat{i}}(t) \notin \{0,1\}$, $y_{\hat{i}}(5) \mapsto a(n)=s_n\_s_{(n-1)}\_\dots\_s_2\_s_1 \Rightarrow V(a(n)) \geq n$. We point out that $s_{n+1}=0 \Rightarrow V(a(n))>n$ and $V(a(n))=n \Rightarrow s_{n+1} \neq 0$, since $s_n\_\dots\_s_2\_s_1 \equiv s_{(n+1)}\_s_n\_\dots\_s_2\_s_1 \pmod {10^n} \wedge s_n\_\dots\_s_2\_s_1 \not\equiv s_{(n+1)}\_s_n\_\dots\_s_2\_s_1 \pmod {10^{n+1}}$ is a necessary condition for $V(a(n))=n$.
In particular, we should note that if $y_{(1,3,4,9,10,12,13,15)}$ (\ref{eq:5}) originates all the pentamorphic integers coprime to $10$ satisfying $y^t=y$ (see Proposition \ref{prop6}, Equation (\ref{eq:2})), then 
$y_{(1,3,4,9,10,12)}(5) \mapsto \left(\pm \left(1-2 \cdot 5^{2^n} \right) \pmod {10^n}, \pm \left( 5^{2^n} - 2^{5^n} \right) \pmod {10^n}, \pm \left( 5^{2^n} + 2^{5^n} \right)\pmod {10^n} \right)$ is enough to find all the smallest bases $\bar{a}_{ \left[ 1,3,7,9 \right]}(n) \leq \tilde{a}_{\left[ 1,3,7,9 \right]}(n)$ characterized by a constant congruence speed which is at least equal to any given strictly positive integer $n$. Hence, considering each of the four mentioned congruence classes modulo $10$, the $\bar{a}_{ \left[ 1,3,7,9 \right]}(n)$  (whose constant congruence speed is $V\left(\bar{a}_{ \left[ 1,3,7,9 \right]}(n)\right) \geq n$) are given by Equation (\ref{eq:1}),

\begin{equation} \label{eq:1}
\bar{a}_{ \left[1,3,7,9 \right]}(n) =
\begin{cases}
\left(1- 2 \cdot 5^{2^n} \right)\hspace{-1mm} \pmod {10^n}  \quad \mathrm{iff}  \hspace{2mm}  a \equiv 1 \pmod {10} \wedge a \neq 1 \\
\min_{n} \left( \left( 5^{2^n} - 2^{5^n} \right)\hspace{-1mm} \pmod {10^n},  -\left( 5^{2^n} + 2^{5^n} \right)\hspace{-1mm} \pmod {10^n} \right) \hspace{2mm} \mathrm{iff} \hspace{2mm}  a \equiv 3\pmod {10}  \\
\min_{n} \left( \left(5^{2^n} + 2^{5^n} \right)\hspace{-1mm} \pmod {10^n}, \left( 2^{5^n} - 5^{2^n} \right)\hspace{-1mm} \pmod {10^n} \right)  \hspace{6mm} \mathrm{iff} \hspace{2mm} a \equiv 7\pmod {10}  \\
 \left(2 \cdot 5^{2^n} - 1 \right)\hspace{-1mm} \pmod {10^n}  \quad  \mathrm{iff} \hspace{2mm}  a \equiv 9\pmod {10}
     \end{cases} \hspace{-5mm}.
\end{equation}

In Equation (\ref{eq:1}), the condition $a \neq 1$ follows from the definition of $V(a)$ itself, which includes $V(1)=0<n$ (for the reasons explained in Reference \cite{Ripa:16}, pages 248–249). Since ${^{b}1}$ is congruent modulo $10^m$ to $^{b+1}1$ for any $m \in \mathbb{N}_0$, the constant congruence speed of $a=1$ is special, and this explains why, in the next proposition, we will exclude $y_{15}(t) : 1^t=1$ from the set of the nontrivial solutions of $y^t=y$.
\end{proposition}

\begin{proposition} \label{prop6}
Let $h(n) \simeq 5^{2^n}$ and $r(n) \simeq 2^{5^n}$, as usual. Assume $t \geq 5$ and let $\{y_i(t), i \in \mathbb{Z}^+ \}$ represent the set of all the solutions in $\mathbb{Z}_{10}$ of the fundamental equation $y^t=y$ (i.e., 
$i \in \{1,2,3,\dots,14,15\}$). Assume that $\alpha'_{s_1} \in \mathbb{Z}_{10}$ and $\alpha''_{s_1} \in \mathbb{Z}_{10}$ (if any) are not equal each other for any $s_1 \in \{1,2,\dots,9\}$ so that we denote with $\{\alpha'_{s_1} \cup \alpha''_{s_1} \} := \{y_{\hat{i}}(t), \hat{i}=1,2,\dots \}$ the subset formed by the $y_i(t)$ which are not congruent modulo $10^2$ to $\{0,1\}$. 
It follows that $\{y_i(5),i=1,\dots,15 \} \supsetneq \{\alpha'_1, \alpha'_2  ,\alpha'_3,\alpha''_3, \alpha'_4, \alpha'_5, \alpha''_5, \alpha'_6, \alpha'_7,\alpha''_7, \alpha'_8, \alpha'_9, \alpha''_9 \}$, since $y_{14} (t) : 0^t=0$ and $y_{15} (t) : 1^t=1$ show the existence of two (trivial) solutions of $y^5=y$ which are not included in the previously mentioned subset. In order to understand how the remaining $y_i(t)$ anticipate the recurrence rules stated in Section \ref{sec:SUBSECTION 2.2}, it can be helpful to preliminary observe that the $y_i(t)$ follow from $\displaystyle\lim_{n\to\infty} 5^{2^n} = \frac{1+\sqrt{1}}{2} \Rightarrow y=\displaystyle\lim_{n\to\infty} 5^{2^n}=\displaystyle\lim_{n\to\infty} 5^{2^{n+1}}=y^2 \Rightarrow y_{j \leq i}(2)=y_{(1,12,14,15)}(t)=\{\alpha'_1,\alpha'_9,0,1\}=\{-\sqrt{1},\sqrt{1},0,1\}$, and we can easily verify that $\alpha'_9= -\alpha'_1=\sqrt{1}=\displaystyle\lim_{n\to\infty} \frac{5^n-2^n}{5^n+2^n}$ \cite{OEIS:7, OEIS:8}. Considering $t=5$, we find in a similar way all the other roots (e.g., see References \cite{OEIS:9, OEIS:10, OEIS:11, OEIS:12} for $\alpha'_3$, $\alpha''_3$, $\alpha'_7$, and $\alpha''_7$), so it is possible to conclude that the $y_{i \leq 13}(t \geq 5)$ are such that $\alpha'_1= -\alpha'_9$, $\alpha'_2=-\alpha'_8$, $\alpha'_3=-\alpha'_7$, $\alpha''_3=-\alpha''_7$, $\alpha'_4=-\alpha'_6$, $\alpha'_5=-\alpha''_5$, and $\alpha''_9=-1$. Furthermore, for any $n$, $r(n)^2+1=h(n) \mapsto 5^{2^n} \equiv \left( \left({2^{5^n}}\right)^2+1 \right) \pmod {10^n}$ if and only if $5^{2^n} \equiv \left(4^{5^n}+1 \right) \pmod {10^n}$.

In general, as clearly explained by Michon in Reference \cite{Michon:6}, we have

\begin{equation} \label{eq:2}
y_{i \leq 13}(t) =
\begin{cases}
\alpha'_1 = 1-2 \cdot h =\dots 538207781991786760045215487480163574218751    \hspace{2.8mm} \textnormal{iff} \hspace{2.8mm} i=1 \\
\alpha'_2 = r =\dots 553032451441224165530407839804103263499879186432    \hspace{2.9mm} \textnormal{iff} \hspace{2.9mm} i=2 \\
\alpha'_3 = h-r =\dots 90779454884838576212137588152996418333704193    \hspace{3.3mm} \textnormal{iff} \hspace{3.3mm} i=3 \\
\alpha''_3 = -h-r =\dots 317662666830362972182803640476581907922943    \hspace{3.6mm} \textnormal{iff} \hspace{3.5mm} i=4 \\
\alpha'_4 = h-1 =\dots 23230896109004106619977392256259918212890624    \hspace{3.0mm} \textnormal{iff} \hspace{3.5mm} i=5 \\
\alpha'_5 = h =\dots 23423230896109004106619977392256259918212890625    \hspace{3.6mm} \textnormal{iff} \hspace{3.6mm} i=6 \\
\alpha''_5 = -h =\dots 6576769103890995893380022607743740081787109375    \hspace{3.0mm} \textnormal{iff} \hspace{2.5mm} i=7 \\
\alpha'_6 = 1-h =\dots 76769103890995893380022607743740081787109376 \hspace{3.3mm} \textnormal{iff} \hspace{3.3mm} i=8 \\
\alpha'_7 = -h+r =\dots 220545115161423787862411847003581666295807    \hspace{3.7mm} \textnormal{iff} \hspace{3.7mm} i=9 \\
\alpha''_7 = h+r =\dots 5682337333169637027817196359523418092077057    \hspace{3.2mm} \textnormal{iff} \hspace{3.2mm} i=10 \\
\alpha'_8 = -r =\dots 967548558775834469592160195896736500120813568    \hspace{3.5mm} \textnormal{iff} \hspace{3.4mm} i=11 \\
\alpha'_9 = 2 \cdot h-1 =\dots 1792218008213239954784512519836425781249    \hspace{3.7mm} \textnormal{iff} \hspace{3.6mm} i=12 \\
\alpha''_9 = -1 =\dots 999999999999999999999999999999999999999999999    \hspace{3.2mm} \textnormal{iff} \hspace{3.1mm} i=13 \\
     \end{cases} \hspace{-5.9mm}.
\end{equation}

Since $\phi : \mathbb{Z}_{10}\rightarrow \frac{\mathbb{Z}}{10^n \mathbb{Z}}$, it follows that $\alpha \mapsto a \pmod {10^n} \Rightarrow V \left(\alpha'_{s_1} \pmod {10^n} \right) \geq n$ and $V\left(\alpha''_{s_1} \pmod {10^n}\right) \geq n$.

More specifically, $\forall n \geq 2, s_{n+1}=0 \Rightarrow \left(V\left(\alpha'_{s_1 \neq 5} \pmod {10^n} \right) \wedge V\left(\alpha''_{s_1 \neq 5} \pmod {10^n} \right)\right) \geq n+1$, while $\left(V\left(\alpha'_5 \pmod {10^n}\right) \wedge V\left(\alpha''_5 \pmod {10^n}\right)\right) \geq n+1$ is true for any $s_{n+1} \in \{0,1,2,3,4,5,6,7,8,9\}$.

In particular, if $\textnormal{gcd}(s_1,10)=1$, then we can easily verify that the relations shown in the next subsection are correct; so, $\forall n \geq 2$, $s_{n+1}\neq 0 \Rightarrow V\left(\alpha'_{(1,3,7,9)} \pmod {10^n} \right)=n$ and also $V\left(\alpha''_{(3,7,9)} \pmod {10^n} \right)=n$.
\end{proposition}

\begin{proposition} \label{prop7}
Let $\alpha'_{s_1}(n) := \alpha'_{s_1} \pmod {10^n}$. Let us consider only the even values of $s_1$ so that $\hat{s}_1 \in \{2,4,6,8\}$. Since $V\left(\alpha'_{s_1} (n)\right) \geq n$ for any $n \in \mathbb{N}-\{0\}$, we only need to compute the residues modulo $2 \cdot 5^n$ of $\alpha'_{\hat{s}_1}$ (observing that $(2 \cdot 5^n)  | \textit{P} \left(V \left(\alpha'_{\hat{s}_1}(n) \right)\right)$ for any $n>1$, see \cite{Ripa:16}, Section 5) in order to find many of the bases $\tilde{a}_{\hat{s}_1}(n)$ that are characterized by a constant congruence speed of $n$ (e.g., if $\hat{{s}_1}=2$ and $n=4$, then $V \left( \alpha'_2(4) \right)=V(6432)=4$, and $V \left(6432 \pmod {2 \cdot 5^4} \right)=V(182)=4=V \left( \tilde{a}_2(4) \right) \Rightarrow \tilde{a}_2(4)=182$). In general, we have that $V \left( \alpha'_{\hat{s}_1}(n) \pmod{2 \cdot 5^n} \right) \geq n$ (e.g., $V \left( \alpha'_2(14) \pmod{2 \cdot 5^n} \right)=15$), and $\alpha'_{\hat{s}_1}(n) \pmod{2 \cdot 5^n}$ always returns the smallest base (congruent modulo $10$ to $\hat{{s}_1}$) which is characterized by a constant congruence speed equal or greater than $n$. Since we are interested in $V \left(\tilde{a}_{\hat{s}_1}(n)\right)=n$ without any exception, we find every $\tilde{a}_{[2,4,6,8]}(n)$ by adding, if necessary, $2 \cdot 5^n$ to $\alpha'_{\hat{s}_1}(n) \pmod {2 \cdot 5^n}$ (e.g., $V \left( \alpha'_8(9)\right)=V(120813568)=9$, and $120813568 \equiv 3626068 \pmod{2 \cdot 5^9}$ would suggest that $\alpha'_8(9) \pmod {2 \cdot 5^9}$ is equal to $3626068$, but clearly $V(3626068)=9+1$) so that $\alpha'_{\hat{s}_1} (n) \pmod {2 \cdot 5^n}+\lambda_{\hat{s}_1} (n) \cdot 2 \cdot 5^n=\tilde{a}_{\hat{s}_1} (n)$ still holds for one $\lambda_{\hat{s}_1} (n) := \lambda\left(\hat{s}_1,n\right) \in \{0,1\}$ (in the two previous examples we verify that $\lambda=1$ holds because $(\alpha'_2 (14) \pmod {2 \cdot 5^n}+1 \cdot 2 \cdot 5^{14}=23316686432 = \tilde{a}_2 (14)>\tilde{a}_2 (15)=\alpha'_2 (14)\pmod {2 \cdot 5^9}$, and also $\alpha'_8 (9) \pmod {2 \cdot 5^n}+1 \cdot 2 \cdot 5^9=7532318= \tilde{a}_8 (9)>\tilde{a}_8 (10)=\alpha'_8 (9) \pmod {2 \cdot 5^9}$). 

In particular, if $\hat{{s}_1} \in \{4,6\}$, then $\tilde{a}_4 (n)=5^n-1 \wedge \tilde{a}_6 (n)=5^n+1$ follows by construction (see $y_{(5,8)} (5)$ by Equation (\ref{eq:2})). Trivially, for any $n$, $5^n-1 \equiv (5^n-1) \pmod{ 2 \cdot 5^n}$ and also $5^n+1 \equiv (5^n+1) \pmod{2 \cdot 5^n}$; thus, $\hat{s}=(4 \vee 6) \Rightarrow \lambda_{(4,6)}=0$ for any positive integer $n$.

Finally, we have that $\lambda_{\hat{s}_1} (n)=1$ if and only if $\hat{{s}_1}=(2 \vee 8) \wedge \alpha'_{\hat{s}_1} (n) \pmod{2 \cdot 5^n}=\alpha'_{\hat{s}_1} (n+1) \pmod{2 \cdot 5^{n+1}}$, while $\lambda=0$ otherwise.

This concludes the proof that, for any $n \geq 1$ and each $\hat{{s}_1} \in \{2,4,6,8\}$, $\exists k(\hat{s}_1,n) \in \mathbb{N}_0 : \alpha'_{\hat{s}_1}(n)-k \cdot 2 \cdot 5^n=\tilde{a}_{\hat{s}_1} (n)$.
\end{proposition}
Lastly, we can find bases congruent to $5$ modulo $10$ that are smaller than $\min_{n} \left(\alpha'_5(n), \alpha''_5 (n) \right)$ and whose constant congruence speed is at least equal to $n$, by simply taking into account that $\textit{P}' \left(\alpha'_5 (n) \right)=\textit{P}' \left(\alpha''_5 (n) \right)=5 \cdot 2^{n+1}$ (see \cite{Ripa:16}, Section 5) and introducing the additional condition $n>2$.

Thus,
\begin{equation} \label{eq:3}
V \left(\alpha'_5 (n) \pmod {10 \cdot 2^n }\right) \geq n  \wedge  V\left(\alpha''_5 (n)\pmod {10 \cdot 2^n} \right) \geq n,
\end{equation}
and Equation (\ref{eq:3}) let us confirm the validity of Equation (\ref{eq:5}) (e.g., if $n=20$, then $\alpha'_5 (20)=92256259918212890625$ is congruent modulo $10 \cdot 2^{20}$ to $9437185$ and $V(9437185)=20$, while $V \left(\alpha''_5 (20) \pmod {10 \cdot 2^{20}}\right)=V(6291455)=21>n$).

\subsection{Main result} \label{sec:SUBSECTION 2.2}

We show that Equation (\ref{eq:4}) is true for any $n \geq 2$ (i.e., $n \geq 2 \Rightarrow \tilde{a}_5 (n)=\tilde{a}(n)$, see Definition \ref{def1.2}).

\begin{equation} \label{eq:4}
\begin{split}
\tilde{a}(n)=\min_{n}\Biggl(2^n \cdot \left(2 \cdot \cos\left( \frac{\pi \cdot (n-1)}{2} \right) -4 \cdot \sin \left( \frac{\pi \cdot (n-1)}{2} \right) +5 \right)+1,  \\   2^n \cdot \left(4 \cdot \sin \left( \frac{\pi \cdot (n-1)}{2} \right) -2 \cdot \cos \left( \frac{\pi \cdot (n-1)}{2} \right) +5 \right)-1 \Biggr).
\end{split}
\end{equation}

Hence,
\begin{equation} \label{eq:5}
\tilde{a}(n)= \begin{cases}
2^n \cdot \left(5+2 \cdot \sin \left( \frac{\pi \cdot n}{2} \right)+4 \cdot \cos \left(\frac{\pi \cdot n}{2} \right) \right)+1   \quad \textnormal{iff} \quad  n \equiv \{2,3\}\pmod{4} \\
2^n \cdot \left(5-2 \cdot \sin \left(\frac{\pi \cdot n}{2} \right)-4 \cdot \cos \left(\frac{\pi \cdot n}{2} \right) \right)-1    \quad \textnormal{iff} \quad n \equiv \{0,1\}\pmod{4}
    \end{cases}\,.
\end{equation}

Now, assume $b>a \geq 2$ (as usual), even if for any $a \equiv \{1,2,4,5,6,8,9\} \pmod {10}$ we are persuaded that $b \geq \textnormal{len}(a)+2$ represents a sufficient condition for $V(a,b)=V(a)$, as predicted by Conjecture \ref{conj1} \cite{Germain:2, Urroz:19}. 
Then, for any given $n \in \mathbb{N}-\{0,1\}$, $V\left(a_{\{s_1\}}(n) \right)=n$, $\forall s_1 \in \{1,2,3,4,5,6,7,8,9\}$, if and only if Equations (\ref{eq:6}), (\ref{eq:7}), (\ref{eq:8}), (\ref{eq:10}), (\ref{eq:11}), (\ref{eq:14}), (\ref{eq:15}), (\ref{eq:16}), and (\ref{eq:17}) are satisfied.

\begin{equation} \label{eq:6}
\begin{split}
\textnormal{A}_1(n)= \begin{cases}
\left(2^{4 \cdot 5^n+1}-1 \right) \pmod{10^n}+j_n \cdot 10^n, 	\\
\forall j_n \not\equiv \frac{\left(2^{4 \cdot 5^{n+1}+1}-1 \right) \pmod{10^{n+1}} - \left( 2^{4 \cdot 5^{n}+1}-1 \right) \pmod{10^n}}{10^n} \pmod{10} \\
\\
10^n+1+k \cdot 10^n, \forall k \equiv \{0,1,2,3,4,5,6,7,8\}\pmod{10}
    \end{cases}\,.
    \end{split}
\end{equation}

With reference to Equation (\ref{eq:6}), we observe that the previously stated condition 
$n \in \mathbb{N}-\{0,1\}$ assures $n \neq 1$ so that we have excluded a priori the possibility that
$\left(2^{4 \cdot 5^{1}+1}-1 	\right) \pmod{10^1}+0 \cdot 10^n=1>0$ gives a contradiction, inasmuch as $V(1)=0$ by definition.

Since $(k+1) \cdot 10^n+1>\left(2^{4 \cdot 5^{n}+1}-1 \right) \pmod{10^n}$ is always true, Equation (\ref{eq:6}) implies that if $n : \left(2^{4 \cdot 5^{n}+1}-1 \right) \pmod{10^n} \not\equiv \left(2^{4 \cdot 5^{n+1}+1}-1 \right) \pmod{10^{n+1}}$, then $\exists! a_1(n) \leq \left(2^{4 \cdot 5^{n}+1}-1 \right) \pmod{10^n}$. Thus, if the $(n+1)$-th rightmost digit of $\alpha'_1$ (see Equation (\ref{eq:2})) is nonzero, then the unique base $a_1(n) \leq \left(2^{4 \cdot 5^{n}+1}-1 \right) \pmod{10^n}$ corresponds to the desired $\tilde{a}_1 (n)$.

In general (as introduced in Proposition \ref{prop6}), $V \left(a_{\{1,9\}}  \right) \leq \min{(\acute{n}, \grave{n})}$, where $\acute{n} : 5^{\acute{n}} \hspace{1mm} | $ \linebreak $\left(a_{\{1,9\}}^2-1 \right) \wedge 5^{\acute{n}+1} \nmid \left(a_{\{1,9\}}^2-1 \right)$, and $\grave{n} : 2^{\grave{n}} | \left(a_{\{1,9\}}^2-1 \right) \wedge 2^{\grave{n}+1} \nmid \left(a_{\{1,9\}}^2-1 \right)$ (i.e., $\acute{n} \neq 0$ is equal to the $5$-adic valuation of $\left(a_{\{1,9\}}^2-1 \right)$, while $\grave{n} \neq 0$ indicates the $2$-adic valuation of $\left(a_{\{1,9\}}^2-1 \right)$).

It follows that, $\forall n \geq 2$, $10^n+1 \geq \tilde{a}_1(n)> \sqrt{5^n+1}$ (since $5^n+1$ is even).

Similarly to Equation (\ref{eq:6}), if $s_1=9$, we have

\begin{equation} \label{eq:7}
\begin{split}
\textnormal{A}_9(n) = \begin{cases}
\left(2 \cdot {5^{2^n}}-1 \right) \pmod{10^n}+j_n \cdot 10^n, 	\\
\forall j_n \not\equiv \frac{\left(2 \cdot {5^{2^{n+1}}}-1 \right)  \pmod{10^{n+1}} - \left(2 \cdot {5^{2^n}}-1 \right) \pmod{10^n}}{10^n} \pmod{10} \\
\\
10^n-1+k \cdot 10^n, \forall k \equiv \{0,1,2,3,4,5,6,7,8\}\pmod{10}
    \end{cases}\,.
    \end{split}
\end{equation}

As previously shown, if $n : \left(2 \cdot 5^{2^n}-1 \right) \pmod{10^n} \not\equiv \left(2 \cdot 5^{2^{n+1}}-1 \right) \pmod{10^{n+1}}$, then 
$\exists! a_9(n) \leq \left(2 \cdot 5^{2^n}-1 \right) \pmod {10^n }$. In general, $V \left(a_{\{1,9\}} \right) \leq \min{(\acute{n}, \grave{n})}$ and Equation (\ref{eq:7}) imply that $10^n> \tilde{a}_9 (n)> \sqrt{5^n+1}$.

We point out that, as a consequence of Proposition \ref{prop6} (see the case $s_{n+1}=0$),

$$n : \frac{\left(2^{4 \cdot 5^{n+1}+1}-1 \right) \pmod{10^{n+1}}-\left(2^{4 \cdot 5^n+1}-1 \right) \pmod{10^n}}{10^n} \equiv 0 \pmod{10}$$

\noindent $\Rightarrow \left(2^{4 \cdot 5^n+1}-1 \right) \equiv \left(2^{4 \cdot 5^{n+1}+1}-1 \right) \pmod{10^{n+1}} \Rightarrow V\left( \left(2^{4 \cdot 5^{n+1}+1}-1 \right) \pmod{10^n} \right)>n$,
and similarly
$$n : \left(2 \cdot 5^{2^{n+1}}-1 \right) \pmod{10^{n+1}}-\left(2 \cdot 5^{2^n}-1 \right) \pmod{10^n} \equiv 0 \pmod{10^{n+1}}$$
$\Rightarrow \left(2 \cdot 5^{2^n}-1 \right) \equiv \left(2 \cdot 5^{2^{n+1}}-1 \right) \pmod{10^{n+1}} \Rightarrow V\left( \left({2 \cdot 5^{2^n}}-1 \right) \pmod{10^n} \right) > n$ \linebreak
(e.g., $V(163574218751)=V\left(\left(2^{4 \cdot 5^{12}+1}-1 \right) \pmod{10^{12}} \right)=13$).

\vspace{2mm}

Let us consider the case $s_1=5$. From \cite{Ripa:16}, we know that, $\forall n \in \mathbb{N}-\{0,1\}$,
\begin{equation} \label{eq:8}
\textnormal{A}_5(n)= \begin{cases}
2^n \cdot \left(5+2 \cdot \sin \left( \frac{\pi \cdot n}{2} \right)+4 \cdot \cos \left(\frac{\pi \cdot n}{2} \right) \right)+1 + k \cdot 5 \cdot 2^{n+1}, \forall k \in \mathbb{N}_0 \\
2^n \cdot \left(5-2 \cdot \sin \left(\frac{\pi \cdot n}{2} \right)-4 \cdot \cos \left(\frac{\pi \cdot n}{2} \right) \right)-1    + k \cdot 5 \cdot 2^{n+1}, \forall k \in \mathbb{N}_0 
    \end{cases}\,.
\end{equation}
Equation (\ref{eq:8}) implies that
\begin{equation} \label{eq:9}
\tilde{a}_5(n) \leq 9 \cdot 2^n + 1,
\end{equation}
\sloppy and the last inequality (trivially) holds because, $\forall n \in \mathbb{N}$, $$\max{\left(\pm x \cdot \cos\left(\frac{\pi}{2} \cdot n \right) \pm y \cdot \sin\left(\frac{\pi}{2} \cdot n \right)    \right)} = \max{(\lvert x \rvert, \lvert y \rvert)}.$$

If $s_1=4$ or $s_1=6$, for the reasons already discussed in the previous subsection, we have, respectively,
\begin{equation} \label{eq:10}
\textnormal{A}_4(n)=5^n - 1 + k \cdot 2 \cdot 5^n, \forall k \equiv \{0, 1, 3, 4\} \pmod {5};
\end{equation}
\begin{equation} \label{eq:11}
\textnormal{A}_6(n)=5^n + 1 + k \cdot 2 \cdot 5^n, \forall k \equiv \{0, 1, 3, 4\} \pmod {5}.
\end{equation}

Equations (\ref{eq:10}) and (\ref{eq:11}) imply that, $\forall n$, $a_{\{4\}}(n)=a_{\{6\}}(n)-2$.

Thus, $\min_n {\left(\tilde{a}_4 (n),\tilde{a}_6 (n) \right)}=\tilde{a}_4 (n)=5^n-1$.

Now, we study the cases $s_1=2$ and $s_1=8$. In general, $V \left(a_{\{2,8\}} \right)$ is less than or equal to $\hat{n}$, the $5$-adic valuation of $\left(a_{\{2,8\}}^2+1 \right)$, and in particular we have
\begin{equation} \label{eq:12}
a_{\{2,8\}}(\hat{n})=\sqrt{5^{\hat{n}} \cdot c_{a_{\{2,8\}}}(\hat{n})-1}=\sqrt{5^{\hat{n}-n} \cdot 5^n \cdot c_{a_{\{2,8\}}}(\hat{n})-1} \Rightarrow V\left(a_{\{2,8\}}(n)\right) = n \leq \hat{n}.
\end{equation}

Since $c_{a_{\{2,8\}}}(\hat{n})$ for any $\hat{n}$, Equation (\ref{eq:12}) states that $\min_{n}{\left( \tilde{a}_2(n), \tilde{a}_8(n) \right)} \geq \sqrt{5^n-1}$.

More specifically, picking any value of $\hat{n}$, the constraint that $c_{a_{\{2,8\}}}=\frac{a_{\{2,8\}}^2+1}{5^{\hat{n}}}$ have to be solved for $c_{a_{\{2,8\}}}$ over the integers (as $a$) let us calculate the solutions (taking the natural logarithm) from
\begin{equation} \label{eq:13}
\hat{n}=\frac{\ln \left( \frac{a_{\{2,8\}}^2+1}{c_{a_{\{2,8\}}}}\right)}{\ln(5)};
\end{equation}
as a random example, we can see that $\hat{n}=20 \Rightarrow \frac{a_{\{2,8\}}^2+1}{5^{20}}=c \in \mathbb{N} \Rightarrow a_{\{2,8\}}(20)=\left(5^{20} \cdot 2 \cdot m+15613890344818 \right) \vee \left(5^{20} \cdot (2 \cdot m+1)+79753541295807 \right)$, where $m \in \mathbb{N}_{0}$. Hence, $m=0 \Rightarrow \tilde{a}_8(20) = 15613890344818$, $\tilde{a}_2(20) = 5^{20} \cdot 1+79753541295807 = 175120972936432$, (since $5^{20} \cdot d+79753541295807$ is odd for any even value of $d$, including zero, while $5^{20} \cdot d+15613890344818$ is odd if and only if $d$ assumes an odd value, and vice versa), and this is enough to conclude that $\tilde{a}_{[2,8]}(20)=\{175120972936432,15613890344818\} \Rightarrow V\left(\tilde{a}_{[2,8]} (20)\right)=20 \Rightarrow V\left(5^{20} \cdot (2 \cdot m+1)+79753541295807\right) \geq 20$ and so is $V\left(5^{20} \cdot 2 \cdot m+15613890344818\right)$ (the last inequality can be proved by observing that $\left(\hat{n}=20,s_1=2,m=1 \right) \Rightarrow a_2 \left(\hat{n}=20,m=1 \right)=\left(5^{20} \cdot (2 \cdot 1+1)+79753541295807 \right)=365855836217682=\tilde{a}_2 (21) \Rightarrow V\left(a_2(20,1)\right)=V\left(a_2 (21,0)\right)=V\left(476837158203125 \cdot 0+365855836217682 \right)=V\left(\tilde{a}_2 (21)\right)=21>20=V\left(\tilde{a}_2 (20)\right)$; ditto for $s_1=8$).

Equation (\ref{eq:13}) provides also a valid upper bound for the constant congruence speed of every element of $\{\textnormal{A}_3 \cup \textnormal{A}_7 \}$, since $a_{\{2,3,7,8\}}^2+1=\prod_{p_j \neq 5}p_j^{q_j} \cdot 5^{\hat{n}} \geq \prod_{p_j \neq 5}p_j^{q_j} \cdot 5^{V\left(a_{\{2,3,7,8\}} \right)} $ (where $p_j$ represents any prime divisor of $a_{\{2,3,7,8\}}^2+1$ which is not equal to $5$, while $q$ indicates how many times the corresponding $p$ appears in the factorization of $a_{\{2,3,7,8\}}^2+1$ \cite{Urroz:19}).

Furthermore, $V\left(a_{\{3,7\}}\right) \leq \min(\hat{n}, \grave{n})$, where $\hat{n} : 5^{\hat{n}} | \left(a_{\{3,7\}}^2+1\right) \wedge 5^{\hat{n}+1} \nmid \left(a_{\{3,7\}}^2+1 \right)$ and $\grave{n} : 2^{\grave{n}} | \left(a_{\{3,7\}}^2-1\right) \wedge 2^{\grave{n}+1} \nmid \left(a_{\{3,7\}}^2-1\right)$. It follows that, for any (strictly) positive integer $n$, $\min_{n}{\left(\tilde{a}_3(n), \tilde{a}_7(n)\right)}>\sqrt{5^n-1}$ (since $5^n-1$ is even).

As shown in Section \ref{sec:SUBSECTION 2.1}, we can easily improve the above upper bound by taking advantage of the commutative ring of $10$-adic integers, giving an explicit formula for $V\left(a_{\{3,7\}}\right)=n$ in the same way as we have already done for $V\left(a_{\{1,9\}}\right)$. For this purpose, let $V\left(a_{\{3,7\}}\right)=n \leq \hat{n}$.

Since $\alpha'_7=h-r=-\alpha'_3$ and $\alpha''_7=h+r=-\alpha''_3$ (where $h(n) \simeq 5^{2^n}$ and $r(n) \simeq 2^{5^n}$), if $s_1=3$, then
\begin{equation} \label{eq:14}
\begin{split}
\textnormal{A}_3(n) = \begin{cases}
\left(5^{2^n}-2^{5^n} \right) \pmod{10^n}+j_n \cdot 10^n, 	\\
\forall j_n \not\equiv \frac{\left({5^{2^{n+1}}}-{2^{5^{n+1}}} \right)  \pmod{10^{n+1}} - \left({5^{2^n}}-2^{5^n} \right) \pmod{10^n}}{10^n} \pmod{10} \\
\\
\left(-5^{2^n}-2^{5^n} \right) \pmod{10^n}+j_n \cdot 10^n, 	\\
\forall j_n \not\equiv \frac{\left({-5^{2^{n+1}}}-{2^{5^{n+1}}} \right)  \pmod{10^{n+1}} - \left({-5^{2^n}}-2^{5^n} \right) \pmod{10^n}}{10^n} \pmod{10}
    \end{cases}\,,
    \end{split}
\end{equation}
while the case $s_1=7$ is covered by Equation (\ref{eq:15})
\begin{equation} \label{eq:15}
\begin{split}
\textnormal{A}_7(n) = \begin{cases}
\left(2^{5^n}-5^{2^n} \right) \pmod{10^n}+j_n \cdot 10^n, 	\\
\forall j_n \not\equiv \frac{\left({2^{5^{n+1}}}-{5^{2^{n+1}}} \right)  \pmod{10^{n+1}} - \left({2^{5^n}}-5^{2^n} \right) \pmod{10^n}}{10^n} \pmod{10} \\
\\
\left(5^{2^n}+2^{5^n} \right) \pmod{10^n}+j_n \cdot 10^n, 	\\
\forall j_n \not\equiv \frac{\left({5^{2^{n+1}}}+{2^{5^{n+1}}} \right)  \pmod{10^{n+1}} - \left({5^{2^n}}+2^{5^n} \right) \pmod{10^n}}{10^n} \pmod{10}
    \end{cases}\,.
    \end{split}
\end{equation}

In order to complete the (constant) congruence speed map, we only need a formula for $\textnormal{A}_2(n)$ and $\textnormal{A}_8(n)$, as shown by Equations (\ref{eq:6}), (\ref{eq:7}), (\ref{eq:8}), (\ref{eq:10}), (\ref{eq:11}), (\ref{eq:14}), and (\ref{eq:15}).

From Proposition \ref{prop7}, we know that $\alpha'_{(2,8)}(n) :=\alpha'_{(2,8)} \pmod {10^n}$ implies
$\alpha'_2(n)+\alpha'_8 (n)=10^n, \alpha'_2 (n) \pmod{2 \cdot 5^n}+\alpha'_8 (n)\pmod{2 \cdot 5^n}=2 \cdot 5^n$, and $\tilde{a}_{(2,8)}(n) = \alpha'_{(2,8)}(n) \pmod{2 \cdot 5^n}+\lambda_{(2,8)}(n) \cdot 2 \cdot 5^n$, where \vspace{-1mm}$$\lambda_{(2,8)}(n) = \begin{cases}
0 \quad \textnormal{iff} \quad \alpha'_{(2,8)}(n) \pmod{2 \cdot 5^n} \neq \alpha'_{(2,8)}(n + 1) \pmod{2 \cdot 5^{n+1}} 	\\
1 \quad \textnormal{iff} \quad \alpha'_{(2,8)}(n) \pmod{2 \cdot 5^n} = \alpha'_{(2,8)}(n + 1) \pmod{2 \cdot 5^{n+1}} 
    \end{cases}\, \hspace{-4mm}.$$

\begin{equation} \label{eq:16}
\textnormal{A}_2(n)=\tilde{a}_2(n)+k \cdot 2 \cdot 5^n, \forall k \not\equiv \frac{\tilde{a}_2(n+1) - \tilde{a}_2(n) }{2 \cdot 5^n}\pmod {5},
\end{equation}

\begin{equation} \label{eq:17}
\textnormal{A}_8(n)=\tilde{a}_8(n)+k \cdot 2 \cdot 5^n, \forall k \not\equiv \frac{\tilde{a}_8(n+1) - \tilde{a}_8(n) }{2 \cdot 5^n}\pmod {5},
\end{equation}
where, as usual, $\tilde{a}_2 (n)=\left(2^{5^n} \pmod{10^n}\right)\pmod{2 \cdot 5^n}+\lambda_2(n) \cdot 2 \cdot 5^n$ and $\tilde{a}_8(n)=\left(-2^{5^n} \pmod{10^n} \right)\pmod{2 \cdot 5^n}+\lambda_8(n) \cdot 2 \cdot 5^n$.

Hence, $\tilde{a}_2 (n)+\tilde{a}_8 (n)=\alpha'_2 (n)\pmod{2 \cdot 5^n }+\alpha'_8 (n)\pmod{2 \cdot 5^n}+2 \cdot 5^n \cdot (\lambda_2 (n)+\lambda_8 (n))$.

Since $\alpha'_2(n) \pmod{2 \cdot 5^n}+\alpha'_8(n) \pmod{2 \cdot 5^n}=2 \cdot 5^n$, for any $n$, we have shown that $\tilde{a}_2(n)=2 \cdot 5^n \cdot (1+\lambda_2(n)+\lambda_8 (n))-\tilde{a}_8(n)$, where $(\lambda_2(n)+\lambda_8(n)) \in \{0,1,2\}$.

\vspace{4mm}
In conclusion, if $V(a)=1$, then \vspace{-2mm}
\begin{equation} \label{eq:18}
a(1) \equiv \{2,3,4,6,8,9,11,12,13,14,16,17,19,21,22,23\}\pmod{25}.
\end{equation}

Therefore, we have mapped all the bases $a$ such that $V(a)=n$.

\vspace{2mm}
The constant congruence speed formula that we have shown in the present section confirms also Hypothesis 1 and Hypothesis 2, stated in Reference \cite{Ripa:16}, as $V(a) \geq 2 \Rightarrow \textit{P} (V(a))=10^{V(a)+1}$ (see \cite{Ripa:15}, Equation (10)).

Now, we are finally ready to prove that $n \geq 2 \Rightarrow \tilde{a}(n)=\tilde{a}_5(n)=2^n \cdot \left((-1)^{n-1}+2 \right)-i^{\hspace{0.5mm} n \cdot (n-1)}$, and this will be the goal of Section \ref{sec:Section3}.


\section{Smallest \bm{$a$} for any given value of the constant congruence speed} \label{sec:Section3}

In this section, we prove the last conjecture stated in \cite{Ripa:16}.

\begin{theorem} \label{Theorem 1}
Let $\tilde{a}(n)$ be defined as in Definition \ref{def1.2}. $\forall n \in \mathbb{N}-\{0,1\}$,
$$\tilde{a}(n) = \begin{cases}
2^n \cdot \left(5+2 \cdot \sin{\left(\frac{\pi \cdot n}{2}\right)}+4 \cdot \cos{\left(\frac{\pi \cdot n}{2}\right)}  \right) +1 \quad \textnormal{iff} \quad n \equiv \{2, 3\}\pmod{4} 	\\

2^n \cdot \left(5-2 \cdot \sin{\left(\frac{\pi \cdot n}{2}\right)}-4 \cdot \cos{\left(\frac{\pi \cdot n}{2}\right)}  \right) -1 \quad \textnormal{iff} \quad n \equiv \{0, 1\}\pmod{4}
    \end{cases}\,;$$
    while $\tilde{a}(1)=2$. Additionally, $\{a(0)\}=\{1\}=\tilde{a}(0)$.
\end{theorem}

\begin{proof}
From Section \ref{sec:SUBSECTION 2.2} (see Equations (\ref{eq:6}) to (\ref{eq:17})), we know that, $\forall n \geq 2$,

$$\tilde{a}_{\{1,9\}}(n) =\min_{n} \left(\tilde{a}_1(n),\tilde{a}_9(n)\right)>\sqrt{5^n+1};$$

$$\tilde{a}_{\{2,8\}}(n) =\min_n \left(\tilde{a}_2(n),\tilde{a}_8(n)\right) \geq \sqrt{5^n-1};$$

$$\tilde{a}_{\{3,7\}}(n) =\min_n \left(\tilde{a}_3(n),\tilde{a}_7(n)\right)>\sqrt{5^n-1};$$

$$\tilde{a}_{\{4,6\}}(n) =\min_n \left(\tilde{a}_4(n),\tilde{a}_6(n)\right)=\tilde{a}_4(n)=5^n-1.$$

Hence,

\begin{equation} \label{eq:19}
\begin{split}
\tilde{a}_{\{1,2,3,4,6,7,8,9\}}(n) = \min_{n}\left(\tilde{a}_{\{1,9\}}(n), \tilde{a}_{\{2,8\}}(n), \tilde{a}_{\{3,7\}}(n), \tilde{a}_{\{4,6\}}(n) \right) \\
\Rightarrow \tilde{a}_{\{1,2,3,4,6,7,8,9\}}(n) \geq \sqrt{5^n-1}.
\end{split}
\end{equation}

On the other hand, Equation (\ref{eq:9}) implies that $\nexists n \in \mathbb{N}-\{{0,1\}} : \tilde{a}_5(n) > 9 \cdot 2^n+1$, since

$$ \tilde{a}_5(n)= \begin{cases}
2^n \cdot \left(2 \cdot \cos \left( \frac{\pi \cdot (n-1)}{2} \right)-4 \cdot \sin \left(\frac{\pi \cdot (n-1)}{2} \right) +5 \right)+1 \quad \textnormal{iff} \quad n \equiv \{2,3\} \pmod{4} \\
2^n \cdot \left(4 \cdot \sin \left( \frac{\pi \cdot (n-1)}{2} \right)-2 \cdot \cos \left(\frac{\pi \cdot (n-1)}{2} \right) +5 \right)-1 \quad \textnormal{iff} \quad n \equiv \{0,1\} \pmod{4}
    \end{cases}\,. $$

Thus, in order to prove the main statement of Theorem \ref{Theorem 1}, it is sufficient to check the inequality $\sqrt{5^n-1}>9 \cdot 2^n+1$, observing that it is certainly true for every $n \geq 20$ (since $\sqrt{5^x-1}=9 \cdot 2^x+1 \Rightarrow 19.693374<x<19.693375$). Consequently, we only need to verify that, $\forall n \in [2,19]$, $\tilde{a}_5 (n)<\tilde{a}_{\{1,2,3,4,6,7,8,9\}}(n)$, and the values are listed in Table \ref{Table:1} (see Equations (\ref{eq:6}) to (\ref{eq:17})).

\begin{table}[ht]
    \centering
    \begin{tabular}{|>{\columncolor[gray]{0.8}}c|c|c|}
        \hline
        \cellcolor{Golden}\boldmath
\bm{$n=V(a)$}&\cellcolor{Mauve}\bm{$\tilde{a}_5(n)$}&\cellcolor{LightCyan}\bm{$\tilde{a}_{\{1,2,3,4,6,7,8,9\}}(n)}$ \\\hline
$1$&$\nexists\tilde{a}_5(1)$&\color{red}$2$\\\hline
$2$&\color{red}$5$&$7$\\\hline
$3$&\color{red}$25$&$57$\\\hline
$4$&\color{red}$15$&$182$\\\hline
$5$&\color{red}$95$&$3124$\\\hline
$6$&\color{red}$65$&$1068$\\\hline
$7$&\color{red}$385$&$32318$\\\hline
$8$&\color{red}$255$&$390624$\\\hline
$9$&\color{red}$1535$&$280182$\\\hline
$10$&\color{red}$1025$&$3626068$\\\hline
$11$&\color{red}$6145$&$23157318$\\\hline
$12$&\color{red}$4095$&$120813568$\\\hline
$13$&\color{red}$24575$&$1220703124$\\\hline
$14$&\color{red}$16385$&$1097376068$\\\hline
$15$&\color{red}$98305$&$11109655182$\\\hline
$16$&\color{red}$65535$&$49925501068$\\\hline
$17$&\color{red}$393215$&$762939453124$\\\hline
$18$&\color{red}$262145$&$355101282318$\\\hline
$19$&\color{red}$1572865$&$19073486328124$
            \\ \hline
\end{tabular}
\caption{Comparison between the smallest $a(n)$ congruent modulo $10$ to $5$, 
whose constant congruence speed is equal to $n \leq 19$, and the minimum $a(n) \equiv \{1,2,3,4,6,7,8,9\} \pmod{10}$. \label{Table:1}}
\end{table}

As it follows from Equations (\ref{eq:9}) and (\ref{eq:19}), $\forall n \in \mathbb{N}-\{0,1\}$, $\tilde{a}(n) := \tilde{a}_{\{1,2,3,4,5,6,7,8,9\}}(n)=\tilde{a}_5(n)$.

Therefore, in order to complete the proof, it is sufficient to observe that $V(2)=1$ and $V(1)=0$ (see \cite{Ripa:16}).
\end{proof}

\begin{corollary} \label{Corollary 1}
Let $\tilde{a}(n)$ be defined as in Definition \ref{def1.2}, and let $i^2=-1$. $\forall n \in \mathbb{N}-\{0,1\}$,
\begin{equation} \label{eq:20}
\tilde{a}(n)=2^n \cdot \left((-1)^{n-1}+2 \right)-i^{\hspace{0.5mm} n \cdot (n-1)}.
\end{equation}
\end{corollary}

\begin{proof}
The statement of Corollary \ref{Corollary 1} easily follows from Theorem \ref{Theorem 1}.

Since, in September 2020, Bruno Berselli noted that Sequence $A337392$ of the OEIS is given by $a(n)=\left(2-(-1)^n \right) \cdot 2^n+i^{(n+1) \cdot (n+2)}$ (see Formula in Reference \cite{OEIS:14}), it trivially follows that Equation (\ref{eq:5}) can be further simplified if we prove the claim;
\begin{equation} \label{eq:21}
\begin{split}
\tilde{a}(n) = \begin{cases}
2^n \cdot \left(5 + 2 \cdot \sin \left( \frac{\pi \cdot n}{2} \right)+4 \cdot \cos \left(\frac{\pi \cdot n}{2} \right) \right)+1 \quad \textnormal{iff} \quad n \equiv \{2,3\} \pmod{4} \\
2^n \cdot \left(5 - 2 \cdot \sin \left( \frac{\pi \cdot n}{2} \right)-4 \cdot \cos \left(\frac{\pi \cdot n}{2} \right)\right)-1 \quad \textnormal{iff} \quad n \equiv \{0,1\} \pmod{4}
\end{cases} &
\\
= 2^{n+1}+ \left(\sin\left(\frac{\pi \cdot (n+1) \cdot (n+2)}{2}\right)-2^n \cdot \sin(\pi \cdot n)\right) \cdot i-2^n \cdot \cos(\pi \cdot n) &
\\
 +\cos\left(\frac{\pi \cdot (n+1) \cdot (n+2)}{2} \right).
\end{split}
\end{equation}

Hence, $2^n \cdot \cos(\pi \cdot n)-i \cdot 2^n \cdot \sin(\pi \cdot n)=-2^n \cdot e^{i \cdot \pi \cdot n}$ implies that
\begin{equation} \label{eq:22}
\tilde{a}(n)=2^{n+1}-2^n \cdot e^{i \cdot \pi \cdot n}+e^{\frac{i \cdot \pi}{2}  \cdot (n \cdot (n+3)+2)}.
\end{equation}

Since $e^{i \cdot \pi}+1=0 \Rightarrow e^{\frac{i \cdot \pi}{2}}=i$ and $e^{i \cdot \pi \cdot n}=(-1)^n$, it follows that
\begin{equation} \label{eq:23}
\tilde{a}(n)=2^{n+1}-2^n \cdot (-1)^n + i^{\hspace{0.5mm} n \cdot (n+3)+2}.
\end{equation}

Thus, Berselli’s formula is correct and we have
\begin{equation} \label{eq:24}
\tilde{a}(n)=2^{n+1}+2^n \cdot (-1)^{n-1} - i^{\hspace{0.5mm} n \cdot (n+3)}.
\end{equation}

Therefore, in order to confirm Equation (\ref{eq:20}) and conclude the proof, it is sufficient to observe that $i^{\hspace{0.5mm} n \cdot (n+3)}=i^{\hspace{0.5mm} n \cdot (n-1)}$.
\end{proof}

\begin{remark} \label{rem2}
Corollary \ref{Corollary 1} provides also a short proof of Theorem \ref{Theorem 1}, since
\begin{equation} \label{eq:25}
\tilde{a}(n)=2^n \cdot \left((-1)^{n-1}+2 \right)-i^{\hspace{0.5mm} n \cdot (n-1)} \leq 2^n \cdot (1+2)+1.
\end{equation}
Thus, $\sqrt{5^n-1}>3 \cdot 2^n+1$ holds for any $n \geq 10$.
\end{remark}

\begin{corollary} \label{Corollary 2}
$\forall n \in \mathbb{N}-\{0,1\}$ and $\forall k \in \mathbb{N}_{0}$,
\begin{equation} \label{eq:26}
\textnormal{A}_5(n)=\left(\left(2^n \cdot \left((-1)^{n-1}+2 \right)-i^{\hspace{0.5mm} n \cdot (n-1)} \right) \vee \left(2^n \cdot \left((-1)^n+8 \right)+i^{\hspace{0.5mm} n \cdot (n-1)} \right)\right)+k \cdot 10 \cdot 2^n.
\end{equation}
\end{corollary}

\begin{proof}
Equation \ref{eq:5} and Corollary \ref{Corollary 1} (Berselli’s formula) imply that
$$\textnormal{A}_5(n)=2^n \cdot \left((-1)^{n-1}+2 \right)-i^{\hspace{0.5mm} n \cdot (n-1)} + k \cdot 10 \cdot 2^n$$
\begin{equation*}
\cup
\begin{cases}
2^n \cdot \left(5 + 2 \cdot \sin \left( \frac{\pi \cdot n}{2} \right)+4 \cdot \cos \left(\frac{\pi \cdot n}{2} \right) \right)+1 + k \cdot 10 \cdot 2^n \quad \textnormal{iff} \quad n \equiv \{0,1\} \pmod{4} \\
2^n \cdot \left(5 - 2 \cdot \sin \left( \frac{\pi \cdot n}{2} \right)-4 \cdot \cos \left(\frac{\pi \cdot n}{2} \right)\right)-1 + k \cdot 10 \cdot 2^n \quad \textnormal{iff} \quad n \equiv \{2,3\} \pmod{4}
\end{cases}.
\end{equation*}

\vspace{3mm}
Since, $\forall n \geq 2$, it easy to verify (as shown in the proof of the aforementioned Corollary \ref{Corollary 1}) that
\begin{equation*}
\begin{cases}
2^n \cdot \left(5 + 2 \cdot \sin \left( \frac{\pi \cdot n}{2} \right)+4 \cdot \cos \left(\frac{\pi \cdot n}{2} \right) \right)+1 \quad \textnormal{iff} \quad n \equiv \{0,1\} \pmod{4} \\
2^n \cdot \left(5 - 2 \cdot \sin \left( \frac{\pi \cdot n}{2} \right)-4 \cdot \cos \left(\frac{\pi \cdot n}{2} \right)\right)-1 \quad \textnormal{iff} \quad n \equiv \{2,3\} \pmod{4}
\end{cases}
\end{equation*}
$$= 2^n \cdot \left(2^3+\cos(\pi \cdot n)+i \cdot \sin(\pi \cdot n) \right)+\cos\left(\frac{\pi \cdot n \cdot (n-1)}{2} \right)+i \cdot \sin\left(\frac{\pi \cdot n \cdot (n-1)}{2} \right) $$
$= 2^n \cdot \left((-1)^n+8 \right)+i^{\hspace{0.5mm} n \cdot (n-1)},$ the statement of Corollary \ref{Corollary 2} follows.
\end{proof}


\section{The constant congruence speed of the prime numbers} \label{sec:Section4}

The set of the prime numbers is very important in many fields of mathematics due to the fundamental theorem of arithmetic (in particular it is central in number theory, computer sciences, and cryptography), so we are interested in knowing if the constant congruence speed of any base which is a prime number maps to every (arbitrarily large) value $V(a) \in \mathbb{N}-\{0\}$.

\begin{definition} \label{def1.3}
Let $a \in \mathbb{N}-\{0\}$ not be a multiple of $10$. $\mathbb{P}=\{p \in \mathbb{N} : p \hspace{2mm} \textit{is prime}\}=\{a : a \hspace{2mm} \textit{is prime}\}$ indicates the set of prime numbers (the last equality holds since $1$ and every multiple of $10$ cannot be prime).
\end{definition}

\begin{definition} \label{def1.4}
${\smalloverbrace{\mathbb{A}}}=\{\left((k+1) \cdot 10^n-1 \right) \in \mathbb{P} : n \in \mathbb{N}-\{0\} \hspace{1mm}\wedge \hspace{1mm} k \in \mathbb{N}_{0}\}$ and \linebreak $\overline{\overline{\mathbb{A}}}=\{\left((2 \cdot k+1) \cdot 10^n-1 \right) \in \mathbb{P} : n \in \mathbb{N}-\{0\} \wedge k \in \mathbb{N}_{0}\}$, so $\overline{\overline{\mathbb{A}}} \subseteq \smalloverbrace{\mathbb{A}} \subset \mathbb{P}$. Additionally, let $\smalloverbrace{a} := (k+1) \cdot 10^n-1$ and let $\overline{\overline{a}} := (2 \cdot k+1) \cdot 10^n-1$. Furthermore, let ${\smalloverbrace{a}}^{*} \in {\smalloverbrace{\mathbb{A}}}$ represents the generic element of $\smalloverbrace{\mathbb{A}}$ and (similarly) let $\overline{\overline{a}}^{*} \in \overline{\overline{\mathbb{A}}}$ be the generic element of $\overline{\overline{\mathbb{A}}}\hspace{0.2mm}$.
\end{definition}

In order to confirm that the set $\{V(p) : p \in \mathbb{P}\}$ is not bounded above, let us firstly introduce the following lemma.

\begin{lemma} \label{Lemma1}
If $b$ is sufficiently large to guarantee $V\left(\smalloverbrace{a},b \right)=V\left(\smalloverbrace{a} \right)$, then $V\left(\smalloverbrace{a} \right) \geq n$, $\forall n \in \mathbb{N}-\{0\}$. In particular, assuming $n \neq 1$, $V\left(\smalloverbrace{a} \right)=n \Leftrightarrow k \not\equiv 9 \pmod {10}$.
\end{lemma}

\begin{proof}
Let $n \geq 2$. The statement easily follows from Equation (\ref{eq:7}), since $n \geq 2 \Rightarrow V(10^n-1+k \cdot 10^n)=n$ for every $k \not\equiv 9 \pmod{10}$, while $k \equiv 9 \pmod{10} \Rightarrow \smalloverbrace{a} \equiv \left(10^{n+1}-1 \right)\pmod {10^{n+1}}$ for any $n$ as above. Thus, $k \equiv 9 \pmod{10} \Leftrightarrow {\smalloverbrace{a}}^{*}$ belongs to $\textnormal{A}_9(n+c)$, where $c \in \mathbb{N}-\{0\}$.

If $n=1$, $\forall k \in \mathbb{N}_{0}$, $\smalloverbrace{a}(k) \equiv \{4,9,14,19,24\} \pmod{25}$. Consequently \cite{Ripa:16}, $V\left(\smalloverbrace{a}(k)\right)>1$ if and only if $\smalloverbrace{a}(k) \equiv 24 \pmod{25}$. 

Hence, $V\left(\smalloverbrace{a}(k)\right) \neq1$ for any $k \equiv 4 \pmod{5}$, including the case $k \equiv 4 \pmod {10}$.

Therefore, for any $n \geq 2$, we have shown that $V\left(\smalloverbrace{a}\right)=n \Leftrightarrow k \not\equiv 9 \pmod{10}$, while 
$n=1 \Rightarrow V\left(\smalloverbrace{a}(k)\right)=n$ if and only if $k \not\equiv 4 \pmod{5}$; since, $\forall n \in \mathbb{N}-\{0\}$, $k \equiv 4 \pmod{5} \Rightarrow V\left(\smalloverbrace{a}(k)\right) \geq n$, the proof of Lemma \ref{Lemma1} is complete.
\end{proof}

\begin{conjecture} \label{conj2}
$\forall \hspace{1mm} {\smalloverbrace{a}}^{*} \in \smalloverbrace{\mathbb{A}}$, $b \geq 2 \Rightarrow V\left({\smalloverbrace{a}}^{*},b\right)=V\left({\smalloverbrace{a}}^{*}\right)$.
\end{conjecture}

\begin{remark} \label{rem3}
In order to show that $\exists {\smalloverbrace{a}}^{*} \in \smalloverbrace{\mathbb{A}} : V\left({\smalloverbrace{a}}^{*},1\right) \neq V\left({\smalloverbrace{a}}^{*},2\right)=V\left({\smalloverbrace{a}}^{*}\right)$, we note that, for any given pair $j \in \mathbb{N}-\{0\}$ and $n \in \mathbb{N}-\{0,1\}$, $V\left((10^n-1)^{5^j},1\right) \neq V\left((10^n-1)^{5^j},2 \right)=V\left((10^n-1)^{5^j} \right)$ (see \cite{Ripa:15}, p. 25). Thus, as a random choice, it is sufficient to take any prime number of the form $c \cdot 10^{10}+(10^2-1)^5$, $c \in \mathbb{N}-\{0\}$, such as $2 \cdot 10^{10}+(10^2-1)^5=295099005 \cdot 10^2-1=29509900499 \in \smalloverbrace{\mathbb{A}}$. Since $V(29509900499,1)=3 \neq 2=V(29509900499,2)$, $b$ have to be greater than $1$.

Checking for smaller candidates than ${\smalloverbrace{a}}^{*}=c \cdot 10^{10}+99^5$, by Hensel’s lifting lemma, we can also see that any odd power of $499 \in \smalloverbrace{\mathbb{A}}$ is congruent modulo $10^3$ to $499$, so $499^{499} \equiv 499 \pmod {10^3}$ and $3=V(499,1) \neq V(499,2)=V(499)=2$ still holds.
\end{remark}

\begin{theorem} \label{Theorem 2}
$\forall a \not\equiv 0 \pmod {10}$, $\exists^{\infty} \hspace{1mm} {\smalloverbrace{a}}^{*} \in \mathbb{P} : V\left({\smalloverbrace{a}}^{*} \right) \geq V(a)$.
\end{theorem}

\begin{proof}
Since $V(a)$ indicates the constant congruence speed of $a$ (by Definition \ref{def1.1}, we are allowed to assume the sufficient but not necessary condition $b \geq a+1$), it follows that $a \not\equiv 0 \pmod{10} \Rightarrow V(a) \in \mathbb{N}_{0}$.

Let us invoke Dirichlet’s theorem on arithmetic progressions \cite{Selberg:17, Shapiro:18}, which implies that $\forall (t,d)$ such that $\textnormal{gcd}(t,d)=1$, there is an infinite number of primes of the form $t+m \cdot d$, where $m \in \mathbb{N}_{0}$.

Now, for any $j \in \mathbb{N}-\{0\}$, let $t := t(j)$, and similarly let $d := d(j)$. In particular, assume (without loss of generality) $t(j)=10^j-1$ and $d(j)=10^j$, since it is trivial to point out that 1$0^j=2^j \cdot 5^j$, so $2 \nmid (10^j-1) \wedge 5 \nmid (10^j-1)$.

By Lemma \ref{Lemma1}, we can state that $\smalloverbrace{a}=t(j)+m \cdot d(j)$ is always characterized by a constant congruence speed $V\left(\smalloverbrace{a} \right) \geq j$.

Anyway, in order to clearly show that $V\left(\smalloverbrace{a} \right) \geq j$ holds, let $x_i \in \{0,1,2,3,4,5,6,7,8,9\}$.

Hence,
\begin{equation} \label{eq:27}
\smalloverbrace{a}=\sum_{i=j}^h{x_i} 	\cdot 10^i+10^j-1=x_h \_ \hspace{0.2mm}x_{(h-1)}\_ \dots \_ \hspace{0.2mm}x_{(j+1)}\_ \hspace{0.2mm} x_j\_9\_9\_ \dots \_9\_9
\end{equation}

so that it is evident that $V\left(\smalloverbrace{a} \right)$ only depends on the length of the rightmost repunit ($9$’s) of 
$10^j<\smalloverbrace{a}<10^{h+1}$.

Thus, picking any $a \geq 2$ such that $V(a)$ is arbitrarily large, we have shown that there always exist infinitely many prime numbers ${\smalloverbrace{a}}^{*}\equiv 9 \pmod{10}$ which are characterized by $V\left({\smalloverbrace{a}}^{*}\right) \geq V(a)(x_j \neq 9 \Rightarrow V\left({\smalloverbrace{a}}^{*}\right)=V(a)=j$, see Equation (\ref{eq:27})).

Lastly, $a=1 \Leftrightarrow V(a)=0$ (see Reference \cite{Ripa:16}, Definition 2), so $V\left({\smalloverbrace{a}}^{*}\right)>V(1)$ for any ${\smalloverbrace{a}}^{*} \in \mathbb{P}$.

Therefore, we can write that $\forall V(a) \in \mathbb{N}$, $\exists \hspace{1mm} {\smalloverbrace{a}}^{*} \in \mathbb{P} : V\left({\smalloverbrace{a}}^{*}\right) \geq V(a)$, and this concludes the proof of Theorem \ref{Theorem 2} (since $a \not\equiv 0 \pmod{10} \Rightarrow V(a,b)=V(a) \in \mathbb{N}_{0}$, $\forall b>a$).
\end{proof}

\begin{corollary} \label{Corollary 3}
$\forall V(a) \in \mathbb{N}_{0}$, $\exists^{\infty} \hspace{1mm} {\smalloverbrace{a}}^{*} \in \mathbb{P} : V\left({\smalloverbrace{a}}^{*} \right) > V(a)$.
\end{corollary}

\begin{proof}
In order to prove this corollary of Theorem \ref{Theorem 2}, it is sufficient to take $n=V(a)+1$, so we have $\smalloverbrace{a} = (k+1) \cdot 10^{V(a)+1}-1$ (by Definition \ref{def1.4}).

Thus, $V\left(\smalloverbrace{a}\right) \geq V(a)+1$ is satisfied for any $V(a)$, $k \in \mathbb{N}_{0}$.

It follows that $V\left({\smalloverbrace{a}}^{*}\right)>V(a)$, and (by Dirichlet’s theorem on arithmetic progressions) we know that there are infinitely many bases ${\smalloverbrace{a}}^{*} \in \smalloverbrace{\mathbb{A}} \subset \mathbb{P}$.
\end{proof}

\begin{theorem} \label{Theorem 3}
$\forall a \neq 1 : a \not\equiv 0 \pmod{10}$, $\exists^{\infty} \hspace{1mm} {\overline{\overline{a}}}^{*} \in \mathbb{P} : V\left({\overline{\overline{a}}}^{*} \right) = V(a)$.
\end{theorem}

\begin{proof}
$\left(a \neq 1 \wedge a \not\equiv 0 \pmod{10}\right) \Rightarrow V(a) \in \mathbb{N}-\{0\}$ \cite{Ripa:16}. By Dirichlet’s theorem on arithmetic progressions, $\textnormal{gcd}(t,d) = 1 \Rightarrow \exists^{\infty}(t+m \cdot d) \in \mathbb{P}$, where $m \in \mathbb{N}_{0}$. Let $t(j)=10^j-1$ and $d(j)=2 \cdot 10^j$ be defined for every $j \in \mathbb{N}-\{0\}$. Since $2 \cdot 10^j=2^{j+1} \cdot 5^j$ and $10^j-1 \equiv 9 \pmod{10}, \textnormal{gcd}\left(t(j),d(j)\right)=1$ (noticing again that $2 \nmid (10^j-1) \wedge 5 \nmid (10^j-1)$).

Consequently, the arithmetic progression $t(j)+m \cdot d(j)=10^j-1+2 \cdot m \cdot 10^j=(2 \cdot m+1) \cdot 10^j-1$ contains infinitely many primes. Since, $\forall m$, $10 \nmid (2 \cdot m+1)$, $t(j)+m \cdot d(j)=\sum_{i=j}^{h} x_i \cdot 10^i+10^j-1=x_h\_ \hspace{0.2mm} x_{(h-1)}\_ \dots \_ \hspace{0.2mm} x_{(j+1)} \_ \hspace{0.2mm}x_j \_ 9 \_ 9\_ \dots \_9\_9$, where $x_j \in \{0,2,4,6,8\}$. By Equation (\ref{eq:7}), for any given $j \geq 2$, it follows that $V\left((2 \cdot m+1) \cdot 10^j-1 \right)=j$ holds for all integers $m \geq 0$.

Finally, if $V(a)=1$, then let $m \equiv 1 \pmod{10}$. Dirichlet’s theorem on arithmetic progressions implies the existence of infinitely many primes congruent modulo $100$ to $29$, and we know that all of them have a unitary constant congruence speed (since $29 \equiv 4 \pmod {25}$, $a \equiv 4 \pmod {25} \Rightarrow V(a) = 1$). It follows that $\exists^{\infty} c   \in  \mathbb{N}_0$ such that $\left((c \cdot 20+3 \right) \cdot 10-1)  \in \overline{\overline{\mathbb{A}}}$. Since $V\left((c \cdot 20+3) \cdot 10-1 \right)=1$ for every $c$, we have just proved that $\forall c \in \mathbb{N}_0 : \left((c \cdot 20+3) \cdot 10-1 \right)\in \mathbb{P}$, $\exists \hspace{0.5mm} \overline{\overline{a}}^{*} (c)\in \overline{\overline{\mathbb{A}}} : V\left(\overline{\overline{a}}^{*}(c) \right)=1$.

Thus, $\forall j \hspace{-0.1mm} \in \hspace{-0.1mm} \mathbb{N}-\{0\}$, $ \exists^{\infty} m \hspace{-0.1mm} \in \hspace{-0.1mm} \mathbb{N}_0 : \left((2 \cdot m+1) \cdot 10^j-1 \right) \hspace{-0.1mm} \in \hspace{-0.1mm} \mathbb{P} \hspace{1mm} \wedge \hspace{1mm} V \hspace{-1mm} \left((2 \cdot m+1) \cdot 10^j-1 \right)=j$, and this completes the proof of Theorem \ref{Theorem 3}.
\end{proof}

Theorem \ref{Theorem 3} entails the existence of an infinite sequence of primes, which we indicate as $\{q_n \}$, defined by the smallest prime numbers characterized by a constant congruence speed of $n \in \mathbb{N}-\{0\}$.

More specifically, $\{q_n \}=2, 5, 193, 1249, 22943, 2218751, \dots$ is not a monotonic sequence, because $q_{20}=3640476581907922943 < 23640476581907922943 = q_{19}$, and also $q_{54} = \left(2 \cdot 5^{2^{52}}-1 \right) \pmod{10^{52}} < q_{52} = - \left(5^{2^{52}}+2^{5^{52}} \right) \pmod{10^{52}} < q_{53} = 2 \cdot 10^{53}-1$ (see Table \ref{Table:2}).

\vspace{3mm}

As an exercise, we can try to bound the value of $q_{1762063}$. Since Theorem \ref{Theorem 3} implies that $q_{1762063} \in \mathbb{N}$, let us find a lower bound from the inequalities stated in Section \ref{sec:SUBSECTION 2.2}, and in particular we get $q_{1762063} > \sqrt{5^{1762063}-1}$. From \cite{Caldwell:1}, we know that $9 \cdot 10^{1762063}-1= \sum_{j=1762063}^{1762063} 8 \cdot 10^j+10^{1762063}-1 $ is prime, and $9 \cdot 10^{1762063}-1=9 \cdot 10^{V \left(9 \cdot 10^{1762063}-1 \right)}-1$ has a constant congruence speed of $1762063$. It follows that $\sqrt{5} \cdot 5^{881031}<q_{1762063} \leq 9 \cdot 10^{1762063}-1$ (since $2 \nmid n \in \mathbb{N}-\{0\} \Rightarrow \sqrt{5^n} \notin \mathbb{N}$).

\begin{table}[H]
    \centering
    \begin{tabular}{|>{\columncolor[gray]{0.8}}c|c|}
        \hline
        \cellcolor{Golden}\boldmath
\bm{$n$}&\cellcolor{Carribean}\bm{$\{q_n\}$} \\\hline \unboldmath
$1$&$2$\\\hline
$2$&$5$\\\hline
$3$&$193$\\\hline
$4$&$1249$\\\hline
$5$&$22943$\\\hline
6&$2218751$\\\hline
7&$4218751$\\\hline
8&$74218751$\\\hline
9&$574218751$\\\hline
10&$30000000001$\\\hline
11&$281907922943$\\\hline
12&$581907922943$\\\hline
13&$6581907922943$\\\hline
14&$123418092077057$\\\hline
15&$480163574218751$\\\hline
16&$19523418092077057$\\\hline
17&$40476581907922943$\\\hline
18&$2152996418333704193$\\\hline
19&$23640476581907922943$\\\hline
20&\color{red}$3640476581907922943$\\\hline
21&$803640476581907922943$\\\hline 
$\dots$&$\dots$\\\hline
51&\color{red}$138023544317662666830362972182803640476581907922943$\\\hline
52&$56138023544317662666830362972182803640476581907922943$\\\hline
53&$199999999999999999999999999999999999999999999999999999$\\\hline
54&\color{red}$1114846846461792218008213239954784512519836425781249$
            \\ \hline
\end{tabular}
\caption{${\{q_n \}}$ for $n \leq 21$ and $51 \leq n \leq 54$. Table entries are in red if (and only if) $q_n<q_{n-1}$, so $q_{20}<q_{19}$ (as $q_{51}<q_{50}$) and $q_{54}<q_{53}$ imply that ${\{q_n \}}$ is a non-monotonic sequence of primes. Furthermore, we have also $q_{54}<q_{52}$. \label{Table:2}
}
\end{table}


\section{Conclusion}

$V(a,b)$, the congruence speed of the integer tetration ${^{b}a}$, certainly does not depend on $b$, for any $a \in \mathbb{N}-\{0\}$ which is not a multiple of $10$, if $b$ is larger than $a$ (i.e., the criterion $b>a$ always holds). Thus, let us take any $b=b(a)$ that assures the constancy of the congruence speed of $a$; then Equations (\ref{eq:6}), (\ref{eq:7}), (\ref{eq:10}), (\ref{eq:11}), (\ref{eq:14}), (\ref{eq:15}), (\ref{eq:16}), (\ref{eq:17}), (\ref{eq:18}), and (\ref{eq:26}) return the set of all the bases whose (constant) congruence speed is any given $V(a) \in \mathbb{N}-\{0\}$, and we know from \cite{Ripa:16} that $V(a)=0 \Leftrightarrow a=1$.

Therefore, we can easily determine $\tilde{a} \left( V(a)\right)$, the smallest $a \equiv \{1, 2, 3, 4, 5, 6, 7, 8, 9\} \pmod {10}$ whose constant congruence speed is equal to any given positive integer. Since $\tilde{a}(0)=1$, $\tilde{a}(1)=2$, and $\tilde{a} \left( V(a) : V(a) \geq 2 \right)=2^n \cdot \left( (-1)^{n-1}+2 \right)-i^{n \cdot (n-1)}$ \cite{OEIS:14}, we can finally conclude that the conjecture stated in Reference \cite{Ripa:16} is true.

In Section \ref{sec:Section4}, for any $n \in \mathbb{N}-\{0\}$, we also proved the existence of infinitely many prime numbers with a constant congruence speed of $n$, defining the related sequence $\{q_n \}$ of the smallest primes such that $V \hspace{-0.5mm} \left(q(n) \right)=n$, and consequently showing that $\{q_n \}$ is not monotonic.

In the present paper we have only considered radix-$10$, but our results can be clearly extended to different numeral systems, as shown by \cite{Germain:2} which was inspired by \cite{Urroz:19}; this observation suggests a topic for the next research article.


\section*{Acknowledgements}

The author is very much thankful to the anonymous referees of \textit{Notes on Number Theory and Discrete Mathematics} whose comments and valuable suggestions have substantially improved the correctness and the overall quality of this paper.

\bibliographystyle{unsrt}
\bibliography{The_congruence_speed_formula}

\end{document}